\documentclass{amsart}
\usepackage{amsmath} 
\usepackage{amssymb}
\usepackage{mathrsfs}

\newtheorem{theorem}{Theorem}[section] 
\newtheorem{claim}[theorem]{Claim}

\theoremstyle{definition}
\newtheorem{definition}[theorem]{Definition}

\newtheorem{question}[theorem]{Question}

\theoremstyle{remark}
\newtheorem{remark}[theorem]{Remark}
\newtheorem{notation}[theorem]{Notation}

\newcommand{\rest}{{\restriction}}
 
\newcommand{\Dom}{{\rm Dom}} 
\newcommand{\Min}{{\rm Min}} 
\newcommand{\tr}{{\rm tr}} 
\newcommand{\Th}{{\rm Th}} 
\newcommand{\cd}{{\rm cd}} 

\newcommand{\suc}{{\rm suc}}
\newcommand{\otp}{{\rm otp}}
\newcommand{\Sym}{{\rm Sym}}
\newcommand{\SSy}{{\rm SSy}}

\newcommand{\BA}{{\rm BA}}
\newcommand{\wilog}{{\rm without loss of generality}}

\newcommand{\then}{{\underline{then}}}
\newcommand{\when}{{\underline{when}}}
\newcommand{\Then}{{\underline{Then}}}

\newcommand{\mn}{{\medskip\noindent}}
\newcommand{\sn}{{\smallskip\noindent}}

\newcommand{\cA}{{\mathcal A}}

\newcommand{\bbR}{{\mathbb R}}

\newcommand{\cH}{{\mathcal H}}

\newcommand{\bbL}{{\mathbb L}}

\newcommand{\bbP}{{\mathbb P}}
\newcommand{\bbN}{{\mathbb N}}
\newcommand{\cP}{{\mathcal P}}

\newcommand{\bbQ}{{\mathbb Q}}

\newcount\skewfactor
\def\mathunderaccent#1#2 {\let\theaccent#1\skewfactor#2
\mathpalette\putaccentunder}
\def\putaccentunder#1#2{\oalign{$#1#2$\crcr\hidewidth
\vbox to.2ex{\hbox{$#1\skew\skewfactor\theaccent{}$}\vss}\hidewidth}}
\def\name{\mathunderaccent\tilde-3 }

\newenvironment{PROOF}[2][\proofname.]
   {\begin{proof}[#1]\renewcommand{\qedsymbol}{\textsquare$_{\rm #2}$}}
   {\end{proof}}

\begin{document}

\title {Models of PA: standard sytems without minimal ultrafilters}
\author {Saharon Shelah}
\address{Einstein Institute of Mathematics\\
Edmond J. Safra Campus, Givat Ram\\
The Hebrew University of Jerusalem\\
Jerusalem, 91904, Israel\\
 and \\
 Department of Mathematics\\
 Hill Center - Busch Campus \\ 
 Rutgers, The State University of New Jersey \\
 110 Frelinghuysen Road \\
 Piscataway, NJ 08854-8019 USA}
\email{shelah@math.huji.ac.il}
\urladdr{http://shelah.logic.at}
\thanks{The author thanks Alice Leonhardt for the beautiful typing.
   The author would like to thank
 the Israel Science Foundation (Grant No. 710/07) and the US-Israel
 Binational Science Foundation (Grant No. 2006108) for partial support
 of this research. Paper 944 on author's list}




\date {January 2, 2018}

\begin{abstract}
We prove that $\bbN$, the standard model of arithmetic, 
has an uncountable elementary
extension $N$ such that there is no ultrafilter on the Boolean Algebra
of subsets of $\bbN$ represented in $N$ which is minimal
(i.e. as in Rudin-Keisler order for partitions represented in $N$).
\end{abstract}

\maketitle
\numberwithin{equation}{section}
\setcounter{section}{-1}
\newpage

\section {Introduction} 

Enayat \cite{Ena06p}, Question III, asked (see Definition \ref{0z.7}(1)):
\begin{question}
\label{q0.7}
Can we prove in ZFC that there is an arithmetically
closed ${\cA} \subseteq {\cP}(\omega)$ such that ${\cA}$ carries
no minimal ultrafilter?

He proved the existence of examples, for 
the stronger notion ``2-Ramsey ultrafilter".  In \cite{Sh:937}
we prove that there is an arithmetically closed Borel set 
$\bold B \subseteq {\cP}(\bbN)$ such that any expansion $\bbN^+$ of
$\bbN$ by any uncountably many members of $\bold B$ has
this property, i.e. the family of definable subsets of $\bbN^+$ carries
no 2-Ramsey ultrafilter.

We deal here with Question \ref{q0.7}, proving that there is 
such a family of cardinality $\aleph_1$, this implies the version in
the abstract; (since it it well-known that every arithmetically closed
family of cardinality at most $\aleph_1$ can be realized as the
standard system of some elementary extension of $\bbN$, as shown by
Knight and Nadel \cite{KnNa82}).  We use 
forcing but the result is proved in ZFC.  
On other problems from \cite{Ena06p} see
Enayat-Shelah \cite{EnSh:936} and \cite{Sh:924}, \cite{Sh:937}.

We thank Shimoni Garti and the referee for helpful comments.
\end{question}

\begin{notation}
\label{0z.1}
1) Let pr:$\omega \times \omega \rightarrow \omega$ be
the standard pairing function (i.e. pr$(n,m) = \binom{n+m}{2} +n$,
so one to one onto two-place function).

\noindent
2) Let ${\cA}$ denote a subset of ${\cP}(\omega)$.

\noindent
3) Let BA$({\cA})$ be the Boolean algebra which ${\cA} \cup 
[\omega]^{< \aleph_0}$ generates.

\noindent
4) Let $D$ denote a non-principal ultrafilter on ${\cA}$, meaning
that $D \subseteq {\cA}$ and there is a unique non-principal
ultrafilter $D'$ on the Boolean algebra $\BA({\cA})$ satisfying $D = D' \cap
{\cA}$, notice that in Definition \ref{0z.7} below 
the distinction between an ultrafilter
on ${\cA}$ and on BA$({\cA})$ makes a difference.

\noindent
5) $\tau$ denotes a vocabulary extending $\tau_{\text{PA}} =
\tau_{\bbN} = \{0,1,+,\times,<\}$, usually countable.

\noindent
6) PA$(\tau)$ is Peano arithmetic for the vocabulary $\tau$.
A model $N$ of PA$(\tau)$ is called ordinary if $N \rest
\tau_{\text{PA}}$ extends $\bbN$; usually the models will be ordinary.

\noindent
7) $\varphi(N,\bar a)$ is $\{b:N \models \varphi[b,\bar a]\}$ where
$\varphi(x,\bar y) \in \bbL(\tau_N)$ and $\bar a \in {}^{\ell g(\bar y)}N$.

\noindent
8) $\Sym(A)$ is the set (or group) of permutations of $N$. 

\noindent
9) For sets $u,v$ of ordinals let OP$_{v,u}$, ``the order preserving
function from $u$ to $v$" be defined by: OP$_{v,u}(\alpha) = \beta$
   \underline{iff} $\beta \in v,\alpha \in u$ and otp$(v \cap \beta) = 
\text{ otp}(u \cap \alpha)$.

\noindent
10) We say $u,v \subseteq \text{ Ord}$ form a $\Delta$-system pair
when otp$(u) = \text{ otp}(v)$ and OP$_{v,u}$ is the identity on
$u \cap v$.
\end{notation}

\begin{definition}
\label{0z.2} 
1)  For ${\cA} \subseteq {\cP}(\omega)$ let ar-cl$({\cA}) = 
\{B \subseteq \omega:B$ is first order
defined in $(\bbN,A_1,\dotsc,A_n)$ for some $n < \omega$ and
$A_1,\dotsc,A_n\in {\cA}\}$.  This is called the arithmetic closure
of ${\cA}$.

\noindent
2) For a model $N$ of PA$(\tau)$ let the standard system of $N$,
$\SSy(N)$ be $\{\varphi(M,\bar a) \cap \bbN:\varphi(x,\bar y) \in
\bbL(\tau)$ and $\bar a \in {}^{\ell g(\bar y)}M\}$ so $\subseteq
   {\cP}(\omega)$ for any ordinary model $M$ isomorphic to $N$, see
   \ref{0z.1}(6).
\end{definition}

\begin{definition}
\label{0z.7}  
Let ${\cA} \subseteq {\cP}(\omega)$.

\noindent
0) Let $\cd_0:\cH(\aleph_0) \rightarrow \omega$ be one to one,
and interpreting $\cH(\aleph_0)$ inside $\bbN$
it is (first order) definable by a bounded formula in $\bbN$,
i.e. $\{\cd_0(x,y):x \in y \in \cH(\aleph_0)\}$ is,
and it maps $\bbN
\times \bbN$ into $\bbN$.  For $h \in {}^\omega \omega$ 
let cd$(h) = \{\text{pr}(n,h(n)):n < \omega\}$, 
where pr is the standard pairing function of $\omega$, see
\ref{0z.1}(1) and generally for $H \subseteq \cH(\aleph_0)$ we let
$\cd(H) := \{\cd_0(x):x \in H\}$; this applies, e.g. to $h \in
{}^{[\omega]^k}\omega$. 

\noindent
1) $D$, an ultrafilter on ${\cA}$, is called minimal \when \,: 
if $h \in {}^\omega \omega$ and $\cd(h) \in {\cA}$ then for 
some $X \in D$ we have $h \rest X$ is constant or one-to-one.

\noindent
2) $D$, an ultrafilter on ${\cA}$, is called Ramsey \when  \,: 
if $k < \omega$ and $h:[\omega]^k \rightarrow \{0,1\}$ and cd$(h) \in
{\cA}$ then for some $X \in D$ we have $h \rest [X]^k$ is
constant.  Similarly $k$-Ramsey.

\noindent
3) $D$ a non-principal ultrafilter on ${\cA}$ is called a
$Q$-point \when \, if $h \in {}^\omega \omega$ is increasing and cd$(h)
\in {\cA}$ \then \, for some increasing sequence $\langle n_i:i
 < \omega\rangle$ we have $i < \omega \Rightarrow h(2i) \le n_i <
 h(2i+1)$ and $\{n_i:i < \omega\} \in D$.
\end{definition}

\begin{remark}
In \cite{Sh:937} we also use the following notions:

\noindent
1) $D$ is called 2.5-Ramsey or self-definably closed \when \,: 
if $\bar h = \langle h_i:i < \omega\rangle$ and $h_i \in 
{}^\omega(i+1)$ and cd$(\bar h) = \{\text{cd}(i,\text{cd}(n,h_i(n)):
i < \omega,n < \omega\}$ belongs to ${\cA}$ \then \, for 
some $g \in {}^\omega \omega$ we have:
cd$(g) \in{\cA}$ and $(\forall i)[g(i) \le i \wedge \{n <
\omega:h_i(n) = g(i)\} \in D]$; this follows 
from 3-Ramsey and implies 2-Ramsey.

\noindent
2) $D$ is weakly definably closed \when \,: if $\langle A_i:i <
\omega\rangle$ is a sequence of subsets of $\omega$ and
$\{\text{pr}(n,i):n \in A_i$ and $i < \omega\} 
\in {\cA}$ then $\{i:A_i \in D\} \in D$, (follows from 2-Ramsey).
\end{remark}

\begin{definition}
\label{0z.18} 
1) $\bbL(\bold Q)$ is first order logic when we add the 
quantifier $\bold Q$ where $(\bold Q x) \varphi$
means that there are uncountable many $x$'s satisfying $\varphi$.

\noindent
2) $\bbL_{\omega_1,\omega}(\bold Q)$ is defined parallely.  

See on those logics Keisler \cite{Ke71}.  We shall use Laver forcing
in the proof of Theorem \ref{a1.3}, so let us define this forcing
notion.  
\end{definition}

\begin{definition}
\label{0z.21}
Let $T \subseteq {}^{\omega >}\omega$ be a subtree.  For $a \in
T$ let $\suc_T(a) = \{a \char 94 \langle i \rangle \in T:i \in
\omega\}$.  The trunk $\tr(T)$ of $T$ is a maximal element $a \in T$ such that
$a \le_T b$ or $b \le_T a$ for every $b \in T$.

Such a tree $T$ will be called a Laver tree iff $s = \tr(T)$ and for
every $t \in T$ such that $s \le t$, the set $\suc_T(t)$ is infinite.

We define the forcing notion $\bbQ$ (= Laver forcing) as follows.  A
condition $T \in \bbQ$ is a Laver tree.  If $S,T \in \bbQ$ then $S
\le_{\bbQ} T$ iff $S \supseteq T$.  If $\bold G \subseteq \bbQ$ is generic,
then $\name\eta[\bold G] := 
\{a \in {}^{\omega >}\omega:\exists T \in \bold G,a$ is the trunk of
$T\}$ will be called a Laver real. 
\end{definition}

\begin{claim}
\label{a1.6}  
If $\boxtimes$ then $\boxplus$ where:
\mn
\begin{enumerate}
\item[$\boxtimes$]  $(a) \quad \bar{\bbQ} = \langle
\bbP_\alpha,\name{\bbQ}_\beta:\alpha \le \alpha(*),\beta <
\alpha(*)\rangle$ is a CS iteration
\sn
\item[${{}}$]   $(b) \quad k(*) < \omega$ and $\beta(k) < \alpha(*) <
\omega_1$ for $k<k(*)$
\sn
\item[${{}}$]  $(c) \quad$ each $\name{\bbQ}_\alpha$ is a Laver
forcing (in $\bold V^{\bbP_\alpha}$) and $\name\eta_\alpha$ its generic
\sn
\item[${{}}$]  $(d) \quad h \in ({}^\omega\omega)^{\bold V}$
\sn
\item[${{}}$]  $(e) \quad p \in \bbP_{\alpha(*)}$
\sn
\item[${{}}$]  $(f) \quad p \Vdash_{\bbP_{\alpha(*)}} 
``\name B_k \subseteq \omega$ and $|\name B_k \cap 
[\name \eta_{\beta(k)}(n+1),\name\eta_{\beta(k)}(n+2))|$

\hskip25pt $\le h(\name \eta_{\beta(k)}(n))$ for every $n$ large enough" for
$k<k(*)$
\sn
\item[$\boxplus$]   for some $p_1,p_2$ and $B^*_k$ for $k<k(*)$ we have
\begin{enumerate}
\item[$(a)$]   $\bbP_{\alpha(*)} \models ``p \le p_\ell"$ for $\ell=1,2$
\sn
\item[$(b)$]   $B^*_k \subseteq \omega$ (from $\bold V$)
\sn
\item[$(c)$]  $p_1 \Vdash ``\name B_k \subseteq^* B^*_k"$
\sn
\item[$(d)$]   $p_2 \Vdash ``\name B_k \subseteq^* (\omega \backslash
B^*_k)"$.
\end{enumerate}
\end{enumerate}
\end{claim}

\begin{PROOF}{\ref{a1.6}}
Clearly letting $\name B_* = \cup\{\name B_k:k <
k(*)\}$ we have
\mn
\begin{enumerate}
\item[$(*)$]  $p \Vdash_{\bbP_{\alpha(*)}}$ ``for every large enough $n$ the
set $\name B_* \cap [\name\eta_0(n+1),\name\eta_0(n+2))$ has $\le
\eta_0(n)$ members".  
\end{enumerate}
\mn
Now by the properties of iterating Laver forcing
(\cite{Lv76a} or see \cite[Ch.VI]{Sh:f}), we have:
\mn
\begin{enumerate}
\item[$(*)$]  if $\bold G_1 \subseteq \bbP_1$ is generic
over $\bold V$ and $\eta = \name\eta_0[\bold G_1]$ \then

\begin{equation*}
\begin{array}{clcr}
\Vdash_{\bbP_{\alpha(*)}/\bold G_1} ``&\text{if } \name B \subseteq
\omega \text{ and in } \name B \cap [\eta(n),\eta(n+1)) \\
  &\text{ there are } \le \eta(n)) \text{ elements for every } n
\text{ large enough} \\
  &\text {\then \, for some } B' \in \bold V[\bold G_1],B' \subseteq
\omega,\name B \subseteq B' \text{ and} \\
  &B' \cap [\eta(n),\eta(n+1))) \text{ has } \le (\eta(n))^n \text{
members for every } n \text{ large enough}".
\end{array}
\end{equation*}
\end{enumerate}
\mn
Now this applies in particular to $\name B = \name B_*$ getting $\name B'$.
Hence \wilog \, $\alpha(*)=1$ so we can replace $\bbP_1$ by $\bbQ_0$,
Laver forcing; also for a dense set of $p \in \bbQ_0$ we have: if
$\eta \in p$ is of length $n+1$ so an increasing sequence of natural
numbers, \then \, $p^{[\eta]} := \{\nu \in p:\nu \trianglelefteq \eta$ or $\eta
\trianglelefteq \nu\}$ forces a value $b_\eta$ to $\name B' \cap
[0,\eta(n))$ so necessarily $|b_\eta| \le \eta(n-1)$ when $n>1$.

By thinning $p$, \wilog \, if $\eta \in p$ and $u_\eta = \{n:\eta \char
94 \langle n \rangle \in p\}$ is infinite (equivalently is not a
singleton) then $\langle b_{\eta\char 94 <n>}:n \in u_\eta\rangle$ is
a $\Delta$-system.

The rest of the proof should be easy, too.
\end{PROOF}
\newpage

\section {No minimal ultrafilter on the standard system} 

\begin{theorem}
\label{a1.3}  
Assume that $\bbN_*$ is an expansion of $\bbN$ with countable
vocabulary or $\bbN_*$ is an ordinary model of PA$_\tau$, for some
countable $\tau \supseteq \tau_{\text{PA}}$ such that $\bbN_*$ is 
countable.  \Then \, there is $M$ such that
\mn
\begin{enumerate}
\item[$(a)$]   $\bbN_* \prec M$
\sn
\item[$(b)$]   $\|M\| = \aleph_1$
\sn
\item[$(c)$]   $\SSy(M)$, the standard system of $M$, see
Definition \ref{0z.2}, has no minimal ultrafilter on it, 
see Definition \ref{0z.7}; moreover
\sn
\item[$(d)$]   there is no $Q$-point on $\SSy(M)$
\sn
\item[$(e)$]  $\SSy(M)$ is arithmetically closed.
\end{enumerate}
\end{theorem}

\begin{PROOF}{\ref{a1.3}}
\bigskip

\noindent
\underline{Stage A}:

Without loss of generality $\bbN_*$ is the Skolem Hull of $\emptyset$
as we can expand it by $\aleph_0$ individual constants.

We shall choose a sentence $\psi \in
\bbL_{\omega_1,\omega}(\bold Q)(\tau^*)$ with 
$\tau^* \supseteq \tau(\bbN_*)$ and prove
that it has a model, and for every model $M^+$ of $\psi$, the model $M^+ \rest
\tau(\bbN_*)$ is as required.  By the completeness theorem for 
$\bbL_{\omega_1,\omega}(\bold Q)$ it is enough to prove that $\psi$ has a
model in some forcing extension; of course it is crucial that
$\psi$ can be explicitly defined hence $\in \bold V$.
\bigskip

\noindent
\underline{Stage B}:

Recall $\cd = \cd_0:\cH(\aleph_0) \rightarrow \omega$ be one-to-one onto and
definable in $\bbN$ by a bounded formula in the natural sense; see \ref{0z.7}.

Let $\bold V_0 = \bold V$ and $\lambda = (2^{\aleph_0})^+$.

Let $\bbR_0 = \text{ Levy}(\aleph_1,2^{\aleph_0})$, let $\bold G_0
\subseteq \bbR_0$ be generic over $\bold V_0$ and let $\bold V_1 =
\bold V_0[\bold G_0]$, i.e. in $\bold V_0^{\bbR_0}$ we have CH.

In $\bold V_1$ we have $\lambda = \aleph_2$ and
let $\bbR_1$ be $\bbP_{\omega_2}$ where 
$\bbP_{\omega_2} = \langle \bbP_\alpha,\name {\bbQ}_\beta:\alpha \le
\omega_2,\beta < \omega_2\rangle$ is a CS iteration, each $\bbQ_\alpha$
is a Laver forcing; there are many other
possibilities, let $\name \eta_\alpha \in {}^\omega \omega$ (increasing)
be the $\bbP_{\alpha +1}$-name of the $\name {\bbQ}_\alpha$-generic
real and $\name \nu_\alpha = \langle\text{cd}(\name \eta_\alpha 
\rest n):n < \omega)\rangle$.  
Let $\bold G_1 \subseteq \bbR_1$ be generic over $\bold V_1$ and
$\bold V_2 = \bold V_1[\bold G_1]$ and let $\eta_\alpha = 
\name \eta_\alpha[\bold G_1],\nu_\alpha = \langle\text{cd}(\eta_\alpha
\rest n):n < \omega\rangle = \name \nu_\alpha[\bold G_1]$.

Let $D^2$ be a non-principal ultrafilter on $\omega$ in the universe
$\bold V_2$.
\mn
\begin{enumerate}
\item[$\boxplus_1$]   In the universe $\bold V_2$ 
let $M_1 = \bbN_*^\omega/D^2$, let $a_\alpha = \eta_\alpha/D^2 \in M_1$
\end{enumerate}
\mn
and note
\mn
\begin{enumerate}
\item[$\boxplus_2$]  $\SSy(M_1) = \cP(\bbN)^{\bold V_2}$ hence is
arithmetically closed
\sn
\item[$\boxplus_3$]   let $f_1 \in \bold V_2$ be the function from
$\lambda = \omega^{\bold V_1}_2 = \omega^{\bold V_2}_2$ 
into $M_1$ defined by $f_1(\alpha) = a_\alpha$.
\end{enumerate}
\bigskip

\noindent
\underline{Stage C}:

In $\bold V_1$ (yes, not in $\bold V_2$) let the forcing notion 
$\bbR_2 := \bbP^+_{\omega_2}$ and the set $K$ be defined as follows
(so $\bold B \in \bold V_1$ below, which is equivalent to $\bold B \in
\bold V_0$, similarly for $u$; so in $\boxplus_4(\alpha),\name A$ is a
$\bbP_{\omega_2}$-name):
\mn
\begin{enumerate}
\item[$\boxplus_4$]   $(\alpha) \quad 
K := \{(\alpha,u,\name A):u \subseteq \lambda$ is
 countable, $\alpha \in u,\name A = \bold B(\ldots,\name \eta_\beta,
\ldots)_{\beta \in u}$,

\hskip30pt $\bold B$ a Borel
function from ${}^{\text{otp}(u)}({}^\omega \omega)$ to
${\cP}(\omega)$ such that 

\hskip30pt $\Vdash_{\bbP_{\omega_2}} ``\name A \cap
[\name \eta_\alpha(n+1),\name\eta_\alpha(n+2))$ has $\le \name
\eta_\alpha(n)$ members; moreover 

\hskip30pt $0 = \text{ lim}_n(|\name A \cap [\name \eta_\alpha(n+1),\name
\eta_\alpha(n+2))/ \name \eta_\alpha(n)|"\}$
\sn
\item[${{}}$]   $(\beta) \quad \bold p \in \bbP^+_{\omega_2}$ iff
\begin{enumerate}
\item[${{}}$]   $(a) \quad \bold p = (p,h) = (p_{\bold p},h_{\bold p})$
\sn
\item[${{}}$]  $(b) \quad p \in \bbP_{\omega_2}$
\sn
\item[${{}}$]  $(c) \quad h$ a function from some finite subset $K_{\bold p}$
of $K$ to $\omega_1$
\sn
\item[${{}}$]  $(d) \quad$ if $(\alpha_\ell,u_\ell,\name A_\ell) \in 
K_{\bold p}$ for $\ell=1,2$ and $h(\alpha_1,u_1,\name A_1) = 
h(\alpha_2,u_2,\name A_2)$ 

\hskip25pt and $u_1 \subseteq \alpha_2$ \then \, 
$p \Vdash_{\bbP_{\omega_2}} ``\name A_1 \cap \name A_2$ is finite"
\end{enumerate}
\item[${{}}$]   $(\gamma) \quad 
\bbP^+_{\omega_2} \models \bold p \le \bold q$ iff:
\begin{enumerate}
\item[${{}}$]  $(a) \quad \bbP_{\omega_2} \models p_{\bold p} \le p_{\bold q}$
\sn
\item[${{}}$]   $(b) \quad h_{\bold p} \subseteq h_{\bold q}$.
\end{enumerate}
\end{enumerate}
\mn
Now
\mn
\begin{enumerate}
\item[$(*)_0$]  if $p \in \bbP_{\omega_2},\alpha < \omega_2$ and $p
\Vdash ``\name A \subseteq \omega$ satisfies $\name A \cap
[\name\eta_\alpha(n+1),\name\eta_\alpha(n+2))$ has $\le \name\eta_\alpha(n)$
members for every $n$ large enough and $0=\lim\langle|\name A \cap
[\name\eta_\alpha(n+1),\name\eta_\alpha(n+2))|/\name\eta_\alpha(n):n <
\omega\rangle$" \then \, we can find a triple 
$(q,u,\name A')$ such that
\begin{enumerate}
\item[$(\alpha)$]  $\bbP_{\omega_2} \models ``p \le q"$
\sn
\item[$(\beta)$]   Dom$(q) = u$
\sn
\item[$(\gamma)$]  $u$ a countable set of ordinals $< \lambda$ (in
$\bold V_1$ equivalently in $\bold V_0$)
\sn
\item[$(\delta)$]  $q \Vdash ``\name A = \name A'"$
\sn
\item[$(\varepsilon)$]  $\name A' = \bold
B(\ldots,\name\eta_{\alpha_i},\ldots)_{i < \text{ otp}(u)}$ where
$\alpha_i$ is the $i$-th member of $u$, for some Borel function
${}^{\text{otp}(u)}({}^\omega \omega)$ to $\cP(\omega)$ so $\bold B \in
\bold V_1$ equivalently $\bold V_0$
\sn
\item[$(\zeta)$]  $q(\alpha_i) = \bold
B_i(\ldots,\name\eta_{\alpha_j},\ldots)_{j<i}$ for every $i < \text{
otp}(u)$ for some Borel fucntion $\bold B_i$ from ${}^i({}^\omega \omega)$
to Laver forcing, of course, $\bold B_i$ is from $\bold V_0$.
\end{enumerate}
\end{enumerate}
\mn
[Why?  Standard proof.]
\mn
\begin{enumerate}
\item[$(*)_1$]  $\bbP^+_{\omega_2}$ satisfies the $\aleph_2$-c.c.
\end{enumerate}
\mn
[Why?  We need a property of the iteration $\langle 
\bbP_\alpha,\name {\bbQ}_\beta:\alpha \le \omega_2,\beta <
\omega_2\rangle$ stated in Claim \ref{a1.6}.  In more detail, given a
sequence $\langle \bold p_\alpha:\alpha < \omega_2\rangle$ of members
of $\bbP^+_{\omega_2}$, for each $\alpha < \omega_2$, let 
$\bold p_\alpha = (p_\alpha,h_\alpha)$; and \wilog \, for each
$(\alpha^*_1,u^*_1,\name A^*_1) \in K_{\bold p_\alpha}$ for some
$u^1,\name A^1$, the tuple $(p_\alpha,u,\name A^1)$
is like $(q,u,\name A')$ in $(*)_0,(\beta)-(\zeta)$ and
$(\alpha,u,\name A) \in \Dom(h_\alpha) \Rightarrow u 
\subseteq \text{ Dom}(p_\alpha)$.  Letting
$u_\alpha = \text{ Dom}(p_\alpha)$, we can 
find a stationary $S \subseteq \{\delta <
\omega_2:\text{ cf}(\delta) = \aleph_1\}$ and $p_*,\gamma(*)$ such that:
\mn
\begin{enumerate}
\item[$\bullet$]   $u_\delta \cap \delta = u_*$ for $\delta \in S$ and
  $u_\alpha \subseteq \delta$ for $\alpha < \delta \in S$
\sn
\item[$\bullet$]   $p_\delta \rest \delta \le p_* \in \bbP_\delta$ for
$\delta \in S$
\sn
\item[$\bullet$]   \wilog \, $p_\delta \rest \delta = p_*$ for $\delta
\in S$
\sn
\item[$\bullet$]   otp$(u_\delta) = \gamma(*)$ for $\delta \in S$
\sn
\item[$\bullet$]   if $\delta_1,\delta_2 \in S$ then the order
preserving function OP$_{u_{\delta_2},u_{\delta_1}}$ from
$u_{\delta_1}$ onto $u_{\delta_2}$ maps $\bold p_{\delta_1}$ to 
$\bold p_{\delta_2}$.
\end{enumerate}
\mn
Let $\delta(*) = \Min(S)$ and $\bold G^1_{\delta(*)} \subseteq
\bbP_{\delta(*)}$ be generic over $\bold V_1$ such that $p_* \in \bold
G^1_{\delta(*)}$.  Now we apply the conclusion of Claim 
\ref{a1.6} to $\bbP_{\omega_2}/\bold G_{\delta(*)}$, the rest should
be clear.

For $\delta \in S$, let $\alpha_0(*) = \otp(u_\delta \backslash
\delta_*),\bold h_\delta$ be the order preserving function from
$\alpha_\delta$ onto $u_\delta \backslash \delta$ and
$(p'_\delta,h'_\delta) \in \bbP_{\alpha_\delta}$ be such that $\bold
b_\delta$ maps $(p'_\delta,h'_\delta)$ to $(p_\delta,h_\delta)$.
Clearly $\alpha_\delta,p'_\delta,h'_\delta$ are the same for all
$\delta \in S$ so call them $\alpha(*),p',h'$ and applying \ref{a1.6}
with $p',(\{\alpha,\name A)$: for some $u$ the tuple $(\alpha,u,\name
A)$ belongs to $\Dom(h)\}$ here stands for
$p,\{(\alpha_k,\name\beta_k):k < k(*)\}$ there and get $p'_1,p'_2$ as
there.

Let $\delta_1 < \delta_2$ be from $S$, let $q_{\delta_1}$ be $\bold
h_{\delta_1}(p'_1),q_{\delta_2}$ be $\bold h_{\delta_2}(p'_2)$.
Easily $p_{\delta_\ell} \le q_{\delta_\ell}$ and $q_{\delta_1} \cup
q_{\delta_2}$ is a common upper bound of $p_{\delta_1},p_{\delta_2}$
in $\bbP^+_{w_2}/\bold G^1_{\delta(*)}$.] 
\mn
\begin{enumerate}
\item[$(*)_2$]   $\bbP^+_{\omega_2}$ collapses $\omega_1$ to $\aleph_0$.
\end{enumerate}
\mn
[Why?  Easy but also we can use $\bbP^+_{\omega_2} \times$
Levy$(\aleph_0,\aleph_1)$ instead.]
\mn
\begin{enumerate}
\item[$(*)_3$]   the function $p \mapsto (p,\emptyset)$ is a
complete embedding of $\bbP_{\omega_2}$ into $\bbP^+_{\omega_2}$.
\end{enumerate}
\mn
[Why?  Should be clear.]
\medskip

\noindent
\underline{Stage D}:  Let $\bold G_2 = 
\bold G^+_1 \subseteq \bbP^+_{\omega_2}$ be
generic over $\bold V_1,\bold V_3 = \bold V_1[\bold G_2]$ and by $(*)_3$ \wilog
\, $\bold G_1 = \{p:(p,h) \in \bold G_2\}$.  So $\bold V_3 = \bold
V_1[\bold G_2]$ is a generic extension of 
$\bold V_2$ and let $f_2 = \cup\{h:(p,h) \in \bold G_2\}$.

So
\mn
\begin{enumerate}
\item[$(*)_4$]   in $\bold V_3$ if $f_2(\alpha_1,u_1,\name A_1) =
f_2(\alpha_2,u_2,\name A_2)$ and $u_1 \subseteq \alpha_2$, \then \,
$\name A_1[\bold G_1] \cap \name A_2[\bold G_1]$ is finite.
\end{enumerate}
\medskip

\noindent
In $\bold V_3$ let $M_2$ be an elementary submodel of $({\cH}
(\beth_\omega),\in,\dotsc,\bold V_\ell \cap
\cH(\beth_\omega),\ldots)_{\ell=0,1,2}$ 
of cardinality $\lambda = \aleph^{\bold V_3}_1$ which
includes $\{\alpha:\alpha \le \lambda\} = \{\alpha:\alpha
\le \omega^{\bold V_3}_1\},\{M_1,f_1,f_2,\bold G_0,\bold G_1,\bold G_2\}$ 
and (the universe of) $M_1$, see end of stage B, note that $\|M_2\|
\subseteq |M_2|$.

Let $f_0$ be a one-to-one function from $M_1$ onto $M_2$, let $M_3$ be a
model such that $f_0$ is an isomorphism from $M_1$ onto $M_3$.  Lastly,
let $M_4$ be $M_3$ expanded by $c_0 = \lambda = \omega^{\bold V_1}_2 =
\omega_1^{\bold V_3},c^{M_4}_1 = \omega^{\bold
V}_1,c^{M_4}_2 = M_1,d^{M_4}_{0,\ell} = \bold G_\ell,d_{1,\ell} =
\bbR_\ell,d^{M_4} = \bbN_*,\langle d^{M_4}_{2,n}:n <
\omega \rangle$ list the members of $\bbN_*,Q^{M_4}_0 = |\bbN_*|,\in^{M_2} = 
\in^{\bold V_3} \rest |M_2|,F^M_0 = f_0,F^{M_4}_1 = f_0
\circ f_1$, see end of Stage B, $F^{M_4}_2 = f_2,P^M_\ell = \bold V_\ell 
\cap M_2$ for $\ell = 0,1,2$ (so
$F_\ell$ is a unary function symbol, $P_\ell$ is a unary predicate)
and lastly $<^M_*$, a linear order of $|M_2| = |M_4|$ of order type
$\omega^{\bold V_3}_1$.

We define the sentence $\psi$: it is the conjunction of the following countable
sets and singletons of sentences of $\bbL_{\aleph_1,\aleph_0}(\bold
Q)$ in the vocabulary $\tau(M_4)$ such that $M^+ \models \psi$ iff:
\mn
\begin{enumerate}
\item[$(A)$]   $M^+ \rest \tau(\bbN_*)$ is isomorphic to $\bbN_*$, of
  cousre, $M^+ \rest \tau(\bbN_*)$ has universe $Q^{M^+}_0$
\sn
\item[$(B)$]   $M^+$ is uncountable, moreover $M^+ \models(\bold Q x)$
($x$ an ordinal $< c_0$)
\sn
\item[$(C)$]   $<^{M^+}_*$ is a linear order
\sn
\item[$(D)$]   every proper initial segment by $<^{M^+}_*$ is countable
\sn
\item[$(E)$]   $(|M^+|,\in^{M^+})$ is a model ZFC$^-$ (even a model
of Th$({\cH}(\beth_\omega)^{\bold V_3},\in))$
\sn
\item[$(F)$]  the function $F^{M^+}_1:\{a:M^+ \models 
``a$ an ordinal $< c_0"\} \rightarrow M^+$ is one-to-one
\sn
\item[$(G)$]  $M^+ \models ``K$ is as above"
\sn
\item[$(H)$]   $F^{M^+}_2:K^{M^+} \rightarrow \{a:M \models ``a$ an
ordinal $< c_1"\}$ is as above
\sn
\item[$(I)$]   $M^+ \models$ ``for every $B$ we have 
$B \in \cP(\bbN) \wedge P_2(B)$ iff $B = A
\cap \bbN$ for some definable subset of $A$ in the model $c_2$".
\end{enumerate}
\mn
It is easy to check that
\mn
\begin{enumerate}
\item[$(*)_5$]   $\psi \in \bold V_0$ 
\sn
\item[$(*)_6$]   $M_4 \models \psi$ in $\bold V_3$.
\end{enumerate}
\mn
Hence as the completeness theorem for $\bbL_{\omega_1,\omega}(\bold Q)$
gives absoluteness
\mn
\begin{enumerate}
\item[$(*)_7$]   $\psi$ has a model in $\bold V = \bold V_0$ call it $M_5$.
\end{enumerate}
\mn
By renaming \wilog 
\mn
\begin{enumerate}
\item[$(*)_8$]  $(a) \quad$ if $M_5 \models$ ``$a$ is the $n$-th natural
number" then $a=n$
\sn
\item[${{}}$]  $(b) \quad$ if $M_5 \models ``A \subseteq \omega"$ then
$A = \{n:M^+ \models ``n \in A"\}$
\sn
\item[${{}}$]  $(c) \quad$ if $M_5 \models ``b \in {}^\omega \omega"$
then $b = \{(n_1,n_2):M^+ \models f(n_1) = n_2\}$
\sn
\item[$(*)_9$]   let $N'_* = M_5 \rest \tau(\bbN_*)$, so isomorphic to
  $N_*$, let $N = M \rest \{\in\}$
\sn
\item[$(*)_{10}$]
\begin{enumerate}
\item[(a)]  let $M'_1$ be $c^{M_5}_2$ naturally defined
\sn
\item[(b)]  so $M'_1$ is a model of $\Th(N'_*) = \Th(N_*),N'_* \prec
  M'_1$ and $\|M'_1\| = \aleph_1$
\sn
\item[(c)]  let $\cA$ be $\SSy(M)$, the standard system of $M$
\end{enumerate}
\end{enumerate}
\mn
Clearly
\mn
\begin{enumerate}
\item[$(*)_{11}$]  $(a) \quad N \models$ ``ZC"
\sn
\item[${{}}$]  $(b) \quad M$ is a model of Th$(\bbN_*)$ and $N_* \prec M$
\sn
\item[$(*)_{12}$]  let $\bbR'_\ell = d^{M^+}_{1,\ell}$ and $\bold
G'_\ell = d^{M^+}_{2,\ell}$ and 
let $\bold V'_\ell = (P^{M^+}_\ell,\in^{M^+})$ for $\ell=0,1,2$.
\end{enumerate}
\medskip

\noindent
\underline{Stage E}:  

Clearly $M$ is an uncountable elementary extension of $\bbN_*$, by
clauses (A),(B) of Stage D and \wilog \, $\|M\| = \aleph_1$, 
so $M$ satisfies clauses (a),(b) of
Theorem \ref{a1.3}.  To prove clause (e) recall $\boxplus_2$  and clause
(I) above hence ${\cA} \subseteq {\cP}(\omega)$ is arithmetically
closed; this implies $\cA$ is a Boolean subalgebra.  Also clause (d)
implies clause (c), anyhow to prove them, 
assume toward contradiction that $D$ is an ultrafilter on ${\cA}$
which is minimal or just a $Q$-point.  Let $X = \{a:N \models ``a$ is an
ordinal $< \omega_1"\}$, so $X$ is really an uncountable set.  For each $a
\in X$ define a sequence $\rho_a \in {}^\omega \omega$ by $\rho(n) =
k$ \underline{iff} $M^+ \models ``F_1(a)(n)=k"$.

Clearly $\rho_\alpha$ is an increasing sequence in 
${}^\omega \omega$, hence by the
assumption toward contradiction, there is $A_a \in D \subseteq \cA$
such that $A_a \cap [\rho_a(n+1),\rho_a(n+2))$ has at most one
element (or just $\le \rho_a(n)$ elements) for each $n < \omega$.

So for some element $\name A_a$ of $N,N \models ``\name A_a$, in
$\bold V'_1$, is a $\bbR_1$-name of a subset of $\omega$ and $\name A_a
[\bold G'_1] = A_a"$.

Clearly $M^+ \models$ ``for some countable subset $u$ of $\omega^{\bold
V'_1}_2 = \omega^{\bold V'_3}_1$ from $\bold V'_1$ and 
Borel function $\bold B$ from $\bold V'_1$ we have 
$A_a = \bold B_a(\ldots,\rho_b,\ldots)_{b \in u_a}$ (so some 
$p \in \bold G^+_2$ forces $\name A_a$ satisfies this)".  So using 
$F^{M^+}_2$ there are $a_1 \ne a_2$ from $X$ such
that the parallel of clause $(\beta)(d)$ of stage C holds, see clause (G) of
stage D, so two members of $D$ are almost disjoint, contradiction.
\end{PROOF}

\begin{remark}
\label{a1.9} 
1) Note that in \ref{a1.3} we can replace $\bbQ_0$ by any forcing 
notion similar enough, see \cite{RoSh:470}.

\noindent
2) We can strengthen \ref{a1.3} by replacing ``$Q$-point" by a weaker
   statement.  

Similarly we can weaken the demands on how ``thin" is $\name B$ in
\ref{a1.6} and in the proof of \ref{a1.3}.
\end{remark}
\newpage

\bibliographystyle{alphacolon}
\bibliography{lista,listb,listx,listf,liste,listy,listz}

\def\germ{\frak} \def\scr{\cal} \ifx\documentclass\undefinedcs
  \def\bf{\fam\bffam\tenbf}\def\rm{\fam0\tenrm}\fi 
  \def\defaultdefine#1#2{\expandafter\ifx\csname#1\endcsname\relax
  \expandafter\def\csname#1\endcsname{#2}\fi} \defaultdefine{Bbb}{\bf}
  \defaultdefine{frak}{\bf} \defaultdefine{=}{\B} 
  \defaultdefine{mathfrak}{\frak} \defaultdefine{mathbb}{\bf}
  \defaultdefine{mathcal}{\cal}
  \defaultdefine{beth}{BETH}\defaultdefine{cal}{\bf} \def\bbfI{{\Bbb I}}
  \def\mbox{\hbox} \def\text{\hbox} \def\om{\omega} \def\Cal#1{{\bf #1}}
  \def\pcf{pcf} \defaultdefine{cf}{cf} \defaultdefine{reals}{{\Bbb R}}
  \defaultdefine{real}{{\Bbb R}} \def\restriction{{|}} \def\club{CLUB}
  \def\w{\omega} \def\exist{\exists} \def\se{{\germ se}} \def\bb{{\bf b}}
  \def\equivalence{\equiv} \let\lt< \let\gt>
  \def\implies{\Rightarrow}\def\mathfrak{\bf}\def\germ{\frak} \def\scr{\cal}
  \ifx\documentclass\undefinedcs
  \def\bf{\fam\bffam\tenbf}\def\rm{\fam0\tenrm}\fi 
  \def\defaultdefine#1#2{\expandafter\ifx\csname#1\endcsname\relax
  \expandafter\def\csname#1\endcsname{#2}\fi} \defaultdefine{Bbb}{\bf}
  \defaultdefine{frak}{\bf} \defaultdefine{=}{\B} 
  \defaultdefine{mathfrak}{\frak} \defaultdefine{mathbb}{\bf}
  \defaultdefine{mathcal}{\cal}
  \defaultdefine{beth}{BETH}\defaultdefine{cal}{\bf} \def\bbfI{{\Bbb I}}
  \def\mbox{\hbox} \def\text{\hbox} \def\om{\omega} \def\Cal#1{{\bf #1}}
  \def\pcf{pcf} \defaultdefine{cf}{cf} \defaultdefine{reals}{{\Bbb R}}
  \defaultdefine{real}{{\Bbb R}} \def\restriction{{|}} \def\club{CLUB}
  \def\w{\omega} \def\exist{\exists} \def\se{{\germ se}} \def\bb{{\bf b}}
  \def\equivalence{\equiv} \let\lt< \let\gt>
\providecommand{\bysame}{\leavevmode\hbox to3em{\hrulefill}\thinspace}
\providecommand{\MR}{\relax\ifhmode\unskip\space\fi MR }
\providecommand{\MRhref}[2]{%
  \href{http://www.ams.org/mathscinet-getitem?mr=#1}{#2}
}
\providecommand{\href}[2]{#2}
\begin{thebibliography}{}

\bibitem[Ena08]{Ena06p}
Ali Enayat, \emph{{A standard model of Peano arithmetic with no conservative
  elementary extension}}, Annals of Pure and Applied Logic \textbf{56} (2008),
  308--318.

\bibitem[Kei71]{Ke71}
H.~Jerome Keisler, \emph{{Model theory for infinitary logic. Logic with
  countable conjunctions and finite quantifiers}}, {Studies in Logic and the
  Foundations of Mathematics}, vol.~62, North--Holland Publishing Co.,
  Amsterdam--London, 1971.

\bibitem[KN82]{KnNa82}
J.~Knight and M.~Nadel, \emph{Models of arithmetic and closed ideals}, Journal
  of Symbolic Logic \textbf{47} (1982), 883--840.

\bibitem[KS06]{KoSc06}
R.~Kossak and J.~Schmerl, \emph{{The structure of models of Peano arithmetic}},
  Oxford University Press, 2006.

\bibitem[Lav76]{Lv76a}
Richard Laver, \emph{{On the consistency of Borel's conjecture}}, Acta Math.
  \textbf{137} (1976), 151--169.

\bibitem[Sh:f]{Sh:f}
Saharon Shelah, \emph{{Proper and improper forcing}}, {Perspectives in
  Mathematical Logic}, {Springer}, 1998.

\bibitem[BsSh:242]{BsSh:242}
Andreas Blass and Saharon Shelah, \emph{{There may be simple $P_ {\aleph_ 1}$-
  and $P_ {\aleph_ 2}$-points and the Rudin-Keisler ordering may be downward
  directed}}, {Annals of Pure and Applied Logic} \textbf{33} (1987), 213--243.

\bibitem[RoSh:470]{RoSh:470}
Andrzej Roslanowski and Saharon Shelah, \emph{{Norms on possibilities I:
  forcing with trees and creatures}}, {Memoirs of the American Mathematical
  Society} \textbf{141} (1999), no.~671, xii + 167, arxiv:math.LO/9807172.

\bibitem[Sh:924]{Sh:924}
Saharon Shelah, \emph{{Models of PA: when two elements are necessarily order
  automorphic}}, Mathematical Logic Quarterly \textbf{61} (2015), 399--417,
  arxiv:1004.3342.

\bibitem[EnSh:936]{EnSh:936}
Ali Enayat and Saharon Shelah, \emph{{An improper arithmetically closed Borel
  subalgebra of $P(\omega)$ mod FIN}}, Topology and its Applications
  \textbf{158} (2011), 2495--2502.

\bibitem[Sh:937]{Sh:937}
Saharon Shelah, \emph{{Models of expansions of $\Bbb N$ with no end
  extensions}}, Mathematical Logic Quarterly \textbf{57} (2011), 341--365,
  arxiv:0808.2960.

\end{thebibliography}

\end{document}